\documentclass[11pt]{amsart}
\usepackage{eurosym}
\usepackage{amsmath}
\usepackage{amsfonts}
\usepackage{textcomp}
\usepackage{amssymb}

\newtheorem{theorem}{Theorem}
\theoremstyle{plain}
\newtheorem{acknowledgement}{Acknowledgements}

\newtheorem{lemma}{Lemma}

\newtheorem{proposition}{Proposition}
\newtheorem{remark}{Remark}

\numberwithin{equation}{section}

\begin{document}
\title[Singular quasilinear systems ]{Singular quasilinear elliptic systems
involving gradient terms}

\begin{abstract}
In this paper we establish existence of smooth positive solutions for a
singular quasilinear elliptic system involving gradient terms. The approach
combines sub-supersolutions method and Schauder's fixed point theorem.
\end{abstract}

\author{Pasquale Candito}
\address{Pasquale Candito\\
DIMET, Universit\`{a} degli Studi di Reggio Calabria, 89100 Reggio Calabria,
Italy}
\email{pasquale.candito@unirc.it}
\author{Roberto Livrea}
\address{Roberto Livrea\\
Department of Mathematics and Informatics, University of Palermo, Via
Archirafi, Palermo, Italy}
\email{roberto.livrea@unipa.it}
\author{Abdelkrim Moussaoui}
\address{Abdelkrim Moussaoui\\
Biology Department, A. Mira Bejaia University, Targa Ouzemour, 06000 Bejaia,
Algeria.}
\email{abdelkrim.moussaoui@univ-bejaia.dz}
\maketitle

\section{Introduction}

Let $\Omega \subset 
\mathbb{R}
^{N}$ $\left( N\geq 2\right) $ be a bounded domain with smooth boundary $%
\partial \Omega $. We deal with the following quasilinear elliptic system%
\begin{equation}
\left\{ 
\begin{array}{l}
-\Delta _{p}u=f(x,u,v,\nabla u,\nabla v)\text{ in }\Omega \\ 
-\Delta _{q}v=g(x,u,v,\nabla u,\nabla v)\text{ in }\Omega \\ 
u,v>0\text{ \ \ \ in }\Omega \\ 
u,v=0\text{ \ \ \ on }\partial \Omega%
\end{array}%
\right.  \tag{$P$}  \label{p}
\end{equation}
where $\Delta _{p}$ (resp. $\Delta _{q})$ stands for the $p$-Laplacian
(resp. $q$-Laplacian) differential operator on $W_{0}^{1,p}(\Omega )$ (resp. 
$W_{0}^{1,q}(\Omega )$) with $1<p,q\leq N$. The nonlinearity terms $%
f(x,u,v,\nabla u,\nabla v)$ and $g(x,u,v,\nabla u,\nabla v),$ which is often
expressed as dealing with convection terms, can exhibit singularities when
the variables $u$ and $v$ approach zero. Specifically, we assume that $%
f,g:\Omega\times(0,+\infty )\times (0,+\infty )\times 
\mathbb{R}
^{2N}\rightarrow (0,+\infty )$ are Carath\'{e}odory functions, that is, $%
f(\cdot ,s_{1},s_{2},\xi _{1},\xi _{2})$ and $g(\cdot ,s_{1},s_{2},\xi
_{1},\xi _{2})$ are measurable for every $(s_{1},s_{2},\xi _{1},\xi _{2})\in
(0,+\infty )\times (0,+\infty )\times \mathbb{R}^{2N}$ and $f(x,\cdot ,\cdot
,\cdot ,\cdot )$ and $g(x,\cdot ,\cdot ,\cdot ,\cdot )$ are continuous
functions for a.e. $x\in \Omega ,$ and are subjected to the hypotheses:

\begin{description}
\item[$\mathrm{H}(f)$] There exist constants $M_{1},m_{1}>0$ and $-1<\alpha
_{1}<0<\beta _{1},\gamma _{1},\theta _{1}$ such that%
\begin{equation*}
\begin{array}{c}
m_{1}s_{1}^{\alpha _{1}}s_{2}^{\beta _{1}}\leq f(x,s_{1},s_{2},\xi _{1},\xi
_{2})\leq M_{1}s_{1}^{\alpha _{1}}s_{2}^{\beta _{1}}+|\xi _{1}|^{\gamma
_{1}}+|\xi _{2}|^{\theta _{1}}\ 
\end{array}%
\end{equation*}
for a.e. $x\in \Omega$, for all $s_{1},s_{2}>0$, for all $\xi_1,\xi_2\in%
\mathbb{R}^N$, with%
\begin{equation*}
0\leq \alpha _{1}+\beta _{1}<-\alpha_1+\beta_1<p-1\text{ \ and \ }\max
\{\gamma _{1},\theta _{1}\}<p-1.
\end{equation*}

\item[$\mathrm{H}(g)$] There exist constants $M_{2},m_{2}>0$ and $-1<\beta
_{2}<0<\alpha _{2},\gamma _{2},\theta _{2}$ such that%
\begin{equation*}
\begin{array}{c}
m_{2}s_{1}^{\alpha _{2}}s_{2}^{\beta _{2}}\leq g(x,s_{1},s_{2},\xi _{1},\xi
_{2})\leq M_{2}s_{1}^{\alpha _{2}}s_{2}^{\beta _{2}}+|\xi _{1}|^{\gamma
_{2}}+|\xi _{2}|^{\theta _{2}}%
\end{array}%
\end{equation*}%
for a.e. $x\in \Omega$, for all $s_{1},s_{2}>0$, for all $\xi_1,\xi_2\in%
\mathbb{R}^N$, with 
\begin{equation*}
0\leq \alpha _{2}+\beta _{2}<-\alpha_2+\beta_2<q-1\text{ \ and \ }\max
\{\gamma _{2},\theta _{2}\}<q-1.
\end{equation*}
\end{description}

\bigskip

The main interest of this work lies in the dependence of the right hand side
terms on the solution and its gradient. The presence of the latter
constitutes a serious obstacle in the study of the problem (\ref{p}).
Namely, the imposed hypotheses do not guarantee that the structure of the
system is variational. Thus, variational methods cannot be applied. Another
important aspect of problem (\ref{p}) is that the convection terms can
exhibit singularities when the variables $u$ and $v$ approach zero. This
occur under hypotheses $\mathrm{H}(f)$ and $\mathrm{H}(g)$ where exponents $%
\alpha _{1}$ and $\beta _{2}$ are allowed to be negative. This type of
problem is rare in the literature. Actually, according to our knowledge,
singular system (\ref{p}) was examined only in \cite{MMZ} where the system
is supposed to have a competitive structure. This means that the
nonlinearities $f$ and $g$ are not increasing with respect to $v$ and $u$,
respectively. Beside that, the singularities appear in both the solution and
its gradient through some specific growth conditions. These combined with
properties of the eigenfunction corresponding to the first eigenvalue of the
operators $-\Delta _{p}$ and $-\Delta _{q}$ is a key point on which the
existence results is proved. It is worth pointing out that the assumptions
imposed therein, precisely (1.2)-(1.5), are not satisfied for system (\ref{p}%
) under hypotheses $\mathrm{H}(f)$ and $\mathrm{H}(g)$. Moreover, in this
work, neither competitive nor the complementary situation called cooperative
structure on the system (\ref{p}) is imposed.

The semilinear case (i.e., $p=q=2$) for a class of singular systems with
convection terms was examined by Alves, Carriao and Faria \cite{ACF}, and by
Alves and Moussaoui \cite{AM}, by essentially using the linearity of the
principal part. For singular elliptic systems without gradient terms, we
refer to Alves and Corr\^{e}a \cite{AC}, Alves, Corr\^{e}a and Gon\c{c}alves 
\cite{ACG}, El Manouni, Perera and Shivaji \cite{EPS}, Ghergu \cite{G1,G2},
Hern\'{a}ndez, Mancebo and Vega \cite{HMV}, Montenegro and Suarez \cite{MS},
Motreanu and Moussaoui \cite{MM1, MM2, MM3}.

\bigskip

The main result of the present paper provides the existence of (positive)
smooth solutions for the singular system (\ref{p}).

\begin{theorem}
\label{T1}Assume $\mathrm{H}(f)$ and $\mathrm{H}(g)$ hold. Then problem (\ref%
{p}) admits a (positive) solution $(u,v)$ in $C_{0}^{1}(\overline{\Omega }%
)\times C_{0}^{1}(\overline{\Omega })$ satisfying%
\begin{equation}
\begin{array}{c}
\tilde c_{0}d(x)\leq u(x)\leq \tilde c_{1}d(x)\text{ \ and \ }\tilde
c_{0}^{\prime }d(x)\leq v(x)\leq \tilde c_{1}^{\prime }d(x)\text{ in }\Omega
,%
\end{array}
\label{c}
\end{equation}%
for some positive constants $\tilde c_{0},\tilde c_{0}^{\prime },\tilde
c_{1} $ and $\tilde c_{1}^{\prime }$.
\end{theorem}

The proof of Theorem \ref{T1} is chiefly based on sub-supersolution method
together with Schauder's fixed point Theorem. However, the sub-supersolution
method cannot be directly implemented. On the one hand, this is due to the
presence of singular terms in system (\ref{p}). In this respect, it should
be pointed out that, to the best of our knowledge, there is no theorem
involving gradient terms and singularities which garantees the existence of
a solution within a sub-supersolution pair.\ On the other hand, the
dependence of the right hand side terms on the gradient of the solution
complicates further the application of the above method to the problem (\ref%
{p}). Specifically, the definition of sub-supersolutions pairs for system (%
\ref{p}) (see \cite{CM}) seems to be hardly applicable because, a priori, no
conclusion can be drawn on the comparison of the gradient of two comparable
functions. To handle problem (\ref{p}), we consider an auxiliary system for
which, the sub-supersolution Theorem involving singular terms in \cite[%
Theorem 2.1]{KM} is applicable. Here, we construct the sub and supersolution
pair by choosing suitable functions with an adjustment of adequate
constants. Then, focusing on the rectangle formed by these functions, we
prove the existence of a smallest and a biggest positive solutions of the
auxiliary problem. The argument is based on the Hardy-Sobolev inequality,
Zorn's Lemma and the the $S_{+}$-property of the negative $p$-Laplacian
operator on $W_{0}^{1,p}(\Omega )$. Thereby, these allow to construct a
suitable operator whose fixed points, obtained via Schauder's fixed point
theorem, are exactly solutions of (\ref{p}).

The rest of the paper is organized as follows. Section \ref{S2} presents
auxiliary results related to sub-supersolutions and extremal solutions.
Section \ref{S3} deals with the existence of a smallest positive solution
for an auxialiary system. Section \ref{S4} contains the proof of the main
result.

\section{Preliminary results}

\label{S2}

Given $1<p<+\infty $, the spaces $L^{p}(\Omega )$ and $W_{0}^{1,p}(\Omega )$
are endowed with the usual norms $\Vert u\Vert _{p}=(\int_{\Omega }|u|^{p}\
dx)^{1/p}$ and $\Vert u\Vert _{1,p}=(\int_{\Omega }|\nabla u|^{p}\ dx)^{1/p}$%
, respectively. Denote by $p^{\prime }=\frac{p}{p-1}$ and $q^{\prime }=\frac{%
q}{q-1}$. We will also utilize the spaces $C(\overline{\Omega })$ and $%
C_{0}^{1,\beta }(\overline{\Omega })=\{u\in C^{1,\beta }(\overline{\Omega }%
):u=0\ \mbox{on
$\partial\Omega$}\}$ with $\beta \in (0,1)$.\newline
For later use, we denote by $\lambda _{1,p}$ and $\lambda _{1,q}$ the first
eigenvalue of $-\Delta _{p}$ on $W_{0}^{1,p}(\Omega )$ and of $-\Delta _{q}$
on $W_{0}^{1,q}(\Omega )$, respectively.\newline
Let $\phi _{1,p}$ be the positive eigenfunction of $-\Delta _{p}$
corresponding to $\lambda _{1,p}$, that is $-\Delta _{p}\phi _{1,p}=\lambda
_{1,p}\phi _{1,p}^{p-1}$ \ in $\Omega ,$ \ $\phi _{1,p}=0$ \ on $\partial
\Omega $. Similarly, let $\phi _{1,q}$ be the positive eigenfunction of $%
-\Delta _{q}$ corresponding to $\lambda _{1,q}$, that is $-\Delta _{q}\phi
_{1,q}=\lambda _{1,q}\phi _{1,q}^{q-1}$ \ in $\Omega ,$ $\phi _{1,q}=0$ \ on 
$\partial \Omega $. The strong maximum principle ensures the existence of
positive constants $l_{1}$, $l_{2}$, $\hat l$ and $l$ such that (see also 
\cite{GiaSchiTak}) 
\begin{equation}
l_{1}\phi _{1,p}(x)\leq \phi _{1,q}(x)\leq l_{2}\phi _{1,p}(x)\text{ for all 
}x\in \Omega  \label{10}
\end{equation}%
and%
\begin{equation}
\hat l d(x)\geq\phi _{1,p}(x),\phi _{1,q}(x)\geq ld(x)\text{ for all }x\in
\Omega \text{.}  \label{9}
\end{equation}

Here, $d(x)$ denotes the distance from a point $x\in \overline{\Omega }$ to
the boundary $\partial \Omega $, where $\overline{\Omega }=\Omega \cup
\partial \Omega $ is the closure of $\Omega \subset 
\mathbb{R}
^{N}$.

Throughout the paper, if $(u_1,v_1)$, $(u_2,v_2)\in W_0^{1,p}(\Omega)\times
W_0^{1,q}(\Omega)$ are such that $u_1\leq u_2$ and $v_1\leq v_2$ a.e. in $%
\Omega$ we will write $(u_1,v_1)\leq (u_2,v_2)$ and we will use the notation 
\begin{equation*}
[u_1,u_2]\times[v_1,v_2]=\{(u,v)\in W_0^{1,p}(\Omega)\times
W_0^{1,q}(\Omega):\ u_1\leq u\leq u_2,\ v_1\leq v\leq v_2\ \text{a.e.\ in\ }%
\Omega\}.
\end{equation*}

\bigskip

A (weak) solution of (\ref{p}) is any pair $(u,v)\in W_{0}^{1,p}(\Omega
)\times W_{0}^{1,q}(\Omega )$ such that 
\begin{equation*}
\int_{\Omega }|\nabla u|^{p-2}\nabla u\nabla \varphi \,dx=\int_{\Omega
}f(x,u,v,\nabla u,\nabla v)\varphi \,dx,
\end{equation*}%
\begin{equation*}
\int_{\Omega }|\nabla v|^{q-2}\nabla v\nabla \psi \,dx=\int_{\Omega
}g(x,u,v,\nabla u,\nabla v)\psi \,dx
\end{equation*}%
for all $(\varphi ,\psi )\in W_{0}^{1,p}(\Omega )\times W_{0}^{1,q}(\Omega )$%
.

We will study auxiliary problems with not convection terms, for this reason
let us consider the following quasilinear elliptic problem%
\begin{equation}
\left\{ 
\begin{array}{l}
-\Delta _{p}u=f_{1}(x,u,v)\text{ in }\Omega \\ 
-\Delta _{q}v=f_{2}(x,u,v)\text{ in }\Omega \\ 
u,v>0\text{ in }\Omega \\ 
u,v=0\text{ on }\partial \Omega ,%
\end{array}%
\right.  \tag{$P_{(f_{1},f_{2})}$}  \label{p*}
\end{equation}%
where $f_{i}:\Omega \times (0,+\infty )\times (0,+\infty )\rightarrow 
\mathbb{R}
$, $i=1,2$, are Carath\'{e}odory functions which can exhibit singularities
near zero.

Recall that $\left( \underline{u},\underline{v}\right)$, $\left( \overline{u}%
,\overline{v}\right) \in (W^{1,p}(\Omega )\cap L^{\infty }(\Omega ))\times
(W^{1,q}(\Omega )\cap L^{\infty }(\Omega ))$ form a pair of a
sub-supersolution for (\ref{p*}) if $(\underline{u},\underline v)\leq(%
\overline{u},\overline{v})$ and%
\begin{equation*}
\left\{ 
\begin{array}{l}
\int_{\Omega }\left\vert \nabla \underline{u}\right\vert ^{p-2}\nabla 
\underline{u}\nabla \varphi \ dx-\int_{\Omega }f_{1}(x,\underline{u}%
,w_{2})\varphi \ dx\leq 0 \\ 
\int_{\Omega }\left\vert \nabla \underline{v}\right\vert ^{q-2}\nabla 
\underline{v}\nabla \psi \ dx-\int_{\Omega }f_{2}(x,w_{1},\underline{v})\psi
\ dx\leq 0,%
\end{array}%
\right.
\end{equation*}%
\begin{equation*}
\left\{ 
\begin{array}{l}
\int_{\Omega }\left\vert \nabla \overline{u}\right\vert ^{p-2}\nabla 
\overline{u}\nabla \varphi \ dx-\int_{\Omega }f_{1}(x,\overline{u}%
,w_{2})\varphi \ dx\geq 0 \\ 
\int_{\Omega }\left\vert \nabla \overline{v}\right\vert ^{q-2}\nabla 
\overline{v}\nabla \psi \ dx-\int_{\Omega }f_{2}(x,w_{1},\overline{v})\psi \
dx\geq 0,%
\end{array}%
\right.
\end{equation*}%
for all $(\varphi ,\psi )\in W_{0}^{1,p}\left( \Omega \right) \times
W_{0}^{1,q}\left( \Omega \right) $ with $\varphi ,\psi \geq 0$ a.e. in $%
\Omega $ and for all $\left( w_{1},w_{2}\right) \in[\underline{u},\overline{u%
}]\times[\underline{v},\overline{v}]$.

\begin{lemma}
\label{L5} Let $(\underline{u}_{i},\underline{v}_{i}),(\overline{u}_{i},%
\overline{v}_{i})\in (W^{1,p}(\Omega )\cap L^{\infty }(\Omega ))\times
(W^{1,q}(\Omega )\cap L^{\infty }(\Omega )),$ for $i=1,2$. Put 
\begin{equation*}
\underline u=\max \{\underline{u}_{1},\underline{u}_{2}\},\quad \underline
v=\max \{\underline{v}_{1},\underline{v}_{2}\},
\end{equation*}
\begin{equation*}
\tilde u=\min \{\overline{u}_{1},\overline{u}_{2}\},\quad \tilde v=\min \{%
\overline{v}_{1},\overline{v}_{2}\}
\end{equation*}
and assume that 
\begin{equation*}
\underline u\leq \tilde u,\quad \underline v\leq \tilde v.
\end{equation*}
Moreover, suppose that 
\begin{equation*}
f_{1}(x,w_1,w_2)\in W^{-1,p^{\prime }}(\Omega ),\quad f_{2}(x,w_1,w_2)\in
W^{-1,q^{\prime }}(\Omega )
\end{equation*}
for every $(w_1,w_2)\in[\underline u,\tilde u]\times[\underline v,\tilde v] $
and $(\underline u_i,\underline v_i)$, $(\overline u_i,\overline v_i)$ ($%
i=1,2$) form two pairs of sub-supersolutions for problem (\ref{p*}).\newline
Then $(\underline u,\underline v)$, $(\tilde u,\tilde v)\in (W^{1,p}(\Omega
)\cap L^\infty(\Omega))\times (W^{1,q}(\Omega)\cap L^\infty(\Omega))$ form
also a pair of sub-supersolution for the problem (\ref{p*}).
\end{lemma}

\begin{proof}
Inspired by the proof of \cite[Lemma 3]{MMP}, for a fixed $\varepsilon >0,$
let us define the truncation function $\xi _{\varepsilon }(s)=\max
\{-\varepsilon ,\min \{s,\varepsilon \}\}$ for $s\in 
\mathbb{R}
.$ It is shown in \cite{MM} that $\xi _{\varepsilon }((\overline{u}_{1}-%
\overline{u}_{2})^{-}),\xi _{\varepsilon }((\underline{u}_{1}-\underline{u}%
_{2})^{+})\in W^{1,p}(\Omega ),$%
\begin{equation*}
\nabla \xi _{\varepsilon }((\overline{u}_{1}-\overline{u}_{2})^{-})=\xi
_{\varepsilon }^{\prime }((\overline{u}_{1}-\overline{u}_{2})^{-})\nabla (%
\overline{u}_{1}-\overline{u}_{2})^{-}
\end{equation*}%
and%
\begin{equation*}
\nabla \xi _{\varepsilon }((\underline{u}_{1}-\underline{u}_{2})^{+})=\xi
_{\varepsilon }^{\prime }((\underline{u}_{1}-\underline{u}_{2})^{+})\nabla (%
\underline{u}_{1}-\underline{u}_{2})^{+}.
\end{equation*}%
For any test function $\varphi \in C_{c}^{1}(\Omega )$ with $\varphi \geq 0,$
it holds%
\begin{equation}
\left\langle -\Delta _{p}\overline{u}_{1},\xi _{\varepsilon }((\overline{u}%
_{1}-\overline{u}_{2})^{-})\varphi \right\rangle \geq \int_{\Omega }f_{1}(x,%
\overline{u}_{1},\hat{w}_{2})\xi _{\varepsilon }((\overline{u}_{1}-\overline{%
u}_{2})^{-})\varphi \text{ }dx,  \label{40}
\end{equation}%
\begin{equation}
\left\langle -\Delta _{p}\underline{u}_{1},\xi _{\varepsilon }((\underline{u}%
_{1}-\underline{u}_{2})^{+})\varphi \right\rangle \leq \int_{\Omega }f_{1}(x,%
\underline{u}_{1},\hat{w}_{2})\xi _{\varepsilon }((\underline{u}_{1}-%
\underline{u}_{2})^{+})\varphi \text{ }dx,  \label{40*}
\end{equation}%
for all $\hat{w}_{2}\in W^{1,q}(\Omega )$ with $\underline{v}_{1}\leq \hat{w}%
_{2}\leq \overline{v}_{1}$, and%
\begin{equation}
\left\langle -\Delta _{p}\overline{u}_{2},(\varepsilon -\xi _{\varepsilon }((%
\overline{u}_{1}-\overline{u}_{2})^{-}))\varphi \right\rangle \geq
\int_{\Omega }f_{1}(x,\overline{u}_{2},\check{w}_{2})\left( \varepsilon -\xi
_{\varepsilon }((\overline{u}_{1}-\overline{u}_{2})^{-})\right) \varphi 
\text{ }dx,  \label{41}
\end{equation}%
\begin{equation}
\left\langle -\Delta _{p}\underline{u}_{2},(\varepsilon -\xi _{\varepsilon
}((\underline{u}_{1}-\underline{u}_{2})^{+}))\varphi \right\rangle \leq
\int_{\Omega }f_{1}(x,\underline{u}_{2},\check{w}_{2})\left( \varepsilon
-\xi _{\varepsilon }((\underline{u}_{1}-\underline{u}_{2})^{+})\right)
\varphi \text{ }dx,  \label{41*}
\end{equation}%
for all $\check{w}_{2}\in W^{1,q}(\Omega )$ with $\underline{v}_{2}\leq 
\check{w}_{2}\leq \overline{v}_{2}$. On the other hand, using the
monotonicity of the $p$-Laplacian operator, we get 
\begin{equation}
\begin{array}{l}
\left\langle -\Delta _{p}\overline{u}_{1},\xi _{\varepsilon }((\overline{u}%
_{1}-\overline{u}_{2})^{-})\varphi \right\rangle +\left\langle -\Delta _{p}%
\overline{u}_{2},(\varepsilon -\xi _{\varepsilon }((\overline{u}_{1}-%
\overline{u}_{2})^{-}))\varphi \right\rangle \\ 
\leq \int_{\Omega }|\nabla \overline{u}_{1}|^{p-2}(\nabla \overline{u}%
_{1},\nabla \varphi )_{%
\mathbb{R}
^{N}}\xi _{\varepsilon }((\overline{u}_{1}-\overline{u}_{2})^{-})\text{ }dx
\\ 
+\int_{\Omega }|\nabla \overline{u}_{2}|^{p-2}(\nabla \overline{u}%
_{2},\nabla \varphi )_{%
\mathbb{R}
^{N}}\left( \varepsilon-\xi _{\varepsilon }((\overline{u}_{1}-\overline{u}%
_{2})^{-})\right) \text{ }dx%
\end{array}
\label{42}
\end{equation}%
and%
\begin{equation}
\begin{array}{l}
\left\langle -\Delta _{p}\underline{u}_{1},(\xi _{\varepsilon }((\underline{u%
}_{1}-\underline{u}_{2})^{+}))\varphi \right\rangle +\left\langle -\Delta
_{p}\underline{u}_{2},\varepsilon -\xi _{\varepsilon }((\underline{u}_{1}-%
\underline{u}_{2})^{+})\varphi \right\rangle \\ 
\geq \int_{\Omega }|\nabla \underline{u}_{1}|^{p-2}(\nabla \underline{u}%
_{1},\nabla \varphi )_{%
\mathbb{R}
^{N}}\xi _{\varepsilon }((\underline{u}_{1}-\underline{u}_{2})^{+})\text{ }dx
\\ 
+\int_{\Omega }|\nabla \underline{u}_{2}|^{p-2}(\nabla \underline{u}%
_{2},\nabla \varphi )_{%
\mathbb{R}
^{N}}\left( \varepsilon-\xi _{\varepsilon }((\underline{u}_{1}-\underline{u}%
_{2})^{+})\right) \text{ }dx.%
\end{array}
\label{42*}
\end{equation}%
Then, gathering (\ref{40}) together with (\ref{41}) and (\ref{40*}) together
with (\ref{41*}), by means of (\ref{42}) and (\ref{42*}), one gets 
\begin{equation*}
\begin{array}{l}
\int_{\Omega }|\nabla \overline{u}_{1}|^{p-2}(\nabla \overline{u}_{1},\nabla
\varphi )_{%
\mathbb{R}
^{N}}\frac{1}{\varepsilon }\xi _{\varepsilon }((\overline{u}_{1}-\overline{u}%
_{2})^{-})\text{ }dx \\ 
+\int_{\Omega }|\nabla \overline{u}_{2}|^{p-2}(\nabla \overline{u}%
_{2},\nabla \varphi )_{%
\mathbb{R}
^{N}}\left( 1-\frac{1}{\varepsilon }\xi _{\varepsilon }((\overline{u}_{1}-%
\overline{u}_{2})^{-})\right) \text{ }dx \\ 
\geq \int_{\Omega }f_{1}(x,\overline{u}_{1},w_{2})\frac{1}{\varepsilon }\xi
_{\varepsilon }((\overline{u}_{1}-\overline{u}_{2})^{-})\varphi \text{ }%
dx+\int_{\Omega }f_{1}(x,\overline{u}_{2},w_{2})\left( 1-\frac{1}{%
\varepsilon }\xi _{\varepsilon }((\overline{u}_{1}-\overline{u}%
_{2})^{-})\right) \varphi \text{ }dx,%
\end{array}%
\end{equation*}%
and%
\begin{equation*}
\begin{array}{l}
\int_{\Omega }|\nabla \underline{u}_{1}|^{p-2}(\nabla \underline{u}%
_{1},\nabla \varphi )_{%
\mathbb{R}
^{N}}\frac{1}{\varepsilon }\xi _{\varepsilon }((\underline{u}_{1}-\underline{%
u}_{2})^{+})\text{ }dx \\ 
+\int_{\Omega }|\nabla \underline{u}_{2}|^{p-2}(\nabla \underline{u}%
_{2},\nabla \varphi )_{%
\mathbb{R}
^{N}}\left( 1-\frac{1}{\varepsilon }\xi _{\varepsilon }((\underline{u}_{1}-%
\underline{u}_{2})^{+})\right) \text{ }dx \\ 
\leq \int_{\Omega }f_{1}(x,\underline{u}_{1},w_{2})\frac{1}{\varepsilon }\xi
_{\varepsilon }((\underline{u}_{1}-\underline{u}_{2})^{+})\varphi \text{ }%
dx+\int_{\Omega }f_{1}(x,\underline{u}_{2},w_{2})\left( 1-\frac{1}{%
\varepsilon }\xi _{\varepsilon }((\underline{u}_{1}-\underline{u}%
_{2})^{+})\right) \varphi \text{ }dx,%
\end{array}%
\end{equation*}%
for all $w_{2}\in W^{1,q}(\Omega )$ such that $\underline v\leq w_{2}\leq
\tilde v$ a.e. in $\Omega $. Passing to the limit as $\varepsilon
\rightarrow 0$ and noticing that%
\begin{equation*}
\left\{ 
\begin{array}{c}
\frac{1}{\varepsilon }\xi _{\varepsilon }((\overline{u}_{1}-\overline{u}%
_{2})^{-}\rightarrow \chi _{\{\overline{u}_{1}<\overline{u}_{2}\}}(x) \\ 
\frac{1}{\varepsilon }\xi _{\varepsilon }((\underline{u}_{1}-\underline{u}%
_{2})^{+}\rightarrow \chi _{\{\underline{u}_{1}>\underline{u}_{2}\}}(x)%
\end{array}%
\right. \text{, \ a.e. in }\Omega \text{ as }\varepsilon \rightarrow 0,
\end{equation*}%
where $\chi _{\mathcal{A}}$ is the characteristic function of the set $%
\mathcal{A},$ we obtain%
\begin{equation*}
\int_{\Omega }|\nabla \tilde u|^{p-2}\nabla \tilde u\nabla \varphi \, dx\geq
\int_{\Omega }f_{1}(x,\tilde u,w_{2})\varphi \text{ }dx
\end{equation*}%
and%
\begin{equation*}
\int_{\Omega }|\nabla \underline u|^{p-2}\nabla \underline u\nabla \varphi
\, dx\leq \int_{\Omega }f_{1}(x,\underline u,w_{2})\varphi \text{ }dx
\end{equation*}%
for all $\varphi \in C_{c}^{1}(\Omega ),$ $\varphi \geq 0$ a.e. in $\Omega $
and for all $w_{2}\in W^{1,q}(\Omega )$ within $[\underline v,\tilde v]$
a.e. in $\Omega $.

In the same manner we get%
\begin{equation*}
\int_{\Omega }|\nabla \tilde v|^{q-2}\nabla \tilde v\nabla \psi\, dx\geq
\int_{\Omega }f_{2}(x,w_{1},\tilde v)\psi \text{ }dx
\end{equation*}%
and%
\begin{equation*}
\int_{\Omega }|\nabla \underline v|^{q-2}\nabla \underline v\nabla \psi \,
dx\leq \int_{\Omega }f_{2}(x,w_{1},\underline v)\psi \, dx
\end{equation*}%
for all $\psi \in C_{c}^{1}(\Omega ),$ $\psi \geq 0$ a.e. in $\Omega $ and
for all $w_{1}\in W^{1,p}(\Omega )$ within $[\underline u,\tilde u]$ a.e. in 
$\Omega $. Finally, since $C_{c}^{1}(\Omega )$ is dense in both $%
W^{1,p}(\Omega )$ and $W^{1,q}(\Omega )$, we achieve the desired conclusion.
\end{proof}

\begin{theorem}
\label{T4} Let $\left( \underline{u},\underline{v}\right) ,$ $\left( 
\overline{u},\overline{v}\right) \in C^{1}(\overline{\Omega })\times C^{1}(%
\overline{\Omega })$ be a pair of sub-supersolution (\ref{p*}) with $%
\underline{u},\underline{v}\geq c_{0}d(x)$ in $\Omega $ for some constant $%
c_{0}>0$ and suppose there exist constants $k_{1},k_{2}>0$ and $-1<\alpha
,\beta <0$ such that%
\begin{equation}
\begin{array}{c}
\left\vert f_{1}(x,u,v)\right\vert \leq k_{1}d(x)^{\alpha }\text{ and }%
\left\vert f_{2}(x,u,v)\right\vert \leq k_{2}d(x)^{\beta }%
\end{array}
\label{h}
\end{equation}
a.e. in $\Omega$ and for every $(u,v)\in [\underline u,\overline u]\times[%
\underline v,\overline v]$. \newline
Then problem (\ref{p*}) has a smallest solution $(u^{\ast },v^{\ast })$ and
a biggest solution $(u^{+},v^{+})$ in $C_{0}^{1,\gamma }(\overline{\Omega }%
)\times C_{0}^{1,\gamma }(\overline{\Omega })$ for certain $\gamma \in (0,1)$%
, within $[\underline{u},\overline{u}]\times \lbrack \underline{v},\overline{%
v}]$.
\end{theorem}

\begin{proof}
We only prove the existence of a smallest positive solution $(u^{\ast
},v^{\ast })\in C_{0}^{1,\gamma }(\overline{\Omega })\times C_{0}^{1,\gamma
}(\overline{\Omega })$ within $[\underline{u},\overline{u}]\times \lbrack 
\underline{v},\overline{v}]$. That of a biggest positive solution within $[%
\underline{u},\overline{u}]\times \lbrack \underline{v},\overline{v}]$ can
be carried out in the similar way.

Denote by $S$ the set of all $(w_{1},w_{2})\in \lbrack \underline{u},%
\overline{u}]\times \lbrack \underline{v},\overline{v}]$ that are solutions
of (\ref{p*}). It is well know from \cite[Theorem 2.1]{KM} that under
assumption (\ref{h}), system (\ref{p*}) has a (positive) solution $(u,v)\in
C_{0}^{1,\gamma }(\overline{\Omega })\times C_{0}^{1,\gamma }(\overline{%
\Omega })$ for certain $\gamma \in (0,1)$, located in $[\underline{u},%
\overline{u}]\times \lbrack \underline{v},\overline{v}].$ Thus, $S$ is not
empty. Moreover, let $(u_{1},v_{1}),(u_{2},v_{2})\in S$. Since $(\underline
u,\underline v)$, $(u_1,v_1)$ and $(\underline u,\underline v)$, $(u_2,v_2)$
form two pairs of sub-supersolution, if we put $(\tilde{u},\tilde{v})=(\min
\{u_{1},u_{2}\},\min \{v_{1},v_{2}\})\in (W_{0}^{1,p}(\Omega )\cap
L^\infty(\Omega))\times (W_{0}^{1,q}(\Omega)\cap L^\infty(\Omega))$, by
virtue of Lemma \ref{L5}, $(\underline u,\underline v)$, $(\tilde u,\tilde
v) $ form a pair of sub-supersolution for (\ref{p*}). Then, owing to \cite[%
Theorem 2.1]{KM}, there exists a solution of (\ref{p*}) in $([\underline{u},%
\tilde{u}]\cap C_{0}^{1}(\overline{\Omega }))\times ([\underline{v},\tilde{v}%
]\cap C_{0}^{1}(\overline{\Omega }))$, which proves that $S$ is downward
directed.

Now, let us consider a chain $C$ in $S$. Then there is a sequence $%
\{(u_{k},v_{k})\}_{k\geq 1}\subset C$ such that $\inf C=\inf_{k\geq
1}(u_{k},v_{k})$ (see \cite[pag. 336]{DunSch}) and it is not restrictive
assume $\{(u_{k},v_{k})\}_{k\geq 1}$ to be decreasing. Hence, if we put $%
(\hat u, \hat v)=\inf C$, one has that $u_k\to \hat u $ and $v_k\to \hat v $
a.e. in $\Omega$, that is 
\begin{equation}  \label{4}
(\hat u,\hat v)\in[\underline u,\overline u]\times [\underline v,\overline
v].
\end{equation}
Moreover, because $(u_{k},v_{k})$ for $k\geq 1$ are solutions of (\ref{p*})
we have%
\begin{equation}
\left\Vert \nabla u_{k}\right\Vert _{p}^{p}=\int_{\Omega
}f_{1}(x,u_{k},v_{k})u_{k}\ dx\leq k_1 \int_{\Omega }d(x)^{\alpha }\overline
u\, dx  \label{1}
\end{equation}%
and%
\begin{equation}
\left\Vert \nabla v_{k}\right\Vert _{q}^{q}=\int_{\Omega
}f_{2}(x,u_{k},v_{k})v_{k}\ dx\leq k_{2}\int_{\Omega }d(x)^{\beta }\overline
v\, dx.  \label{2}
\end{equation}%
Since $-1<\alpha ,\beta <0$, by virtue of the Hardy-Sobolev inequality (see,
e.g., \cite{AC} or \cite{OpiKuf}), the last integrals in (\ref{1}) and (\ref%
{2}) are finite which in turn imply that $\{u_{k}\}$ and $\{v_{k}\}$ are
bounded in $W_{0}^{1,p}(\Omega )$ and $W_{0}^{1,q}(\Omega ),$ respectively.
So, passing to relabelled subsequences and recalling the Rellich embedding
theorem, we have 
\begin{equation}
\begin{array}{c}
(u_{k},v_{k})\rightharpoonup \left( \hat u ,\hat v\right) \text{ in }%
W_{0}^{1,p}(\Omega )\times W_{0}^{1,q}\left( \Omega \right) \text{.}%
\end{array}
\label{3}
\end{equation}%
%
%
%
Using $\varphi =u_{k}-\hat u$ and $\psi =v_{k}-\hat v$ as test functions we
find that 
\begin{equation*}
\begin{array}{l}
\langle -\Delta _{p}u_{k},u_{k}-\hat u\rangle =\int_{\Omega
}f_{1}(x,u_{k},v_{k})(u_{k}-\hat u)\ dx%
\end{array}%
\end{equation*}%
and 
\begin{equation*}
\begin{array}{l}
\langle -\Delta _{q}v_{k},v_{k}-\hat v\rangle =\int_{\Omega
}f_{2}(x,u_{k},v_{k})(v_{k}-\hat v)\ dx.%
\end{array}%
\end{equation*}%
From (\ref{h}) and since $(u_{k},v_{k})\in S$ for all $k\in 
\mathbb{N}
$, in view of (\ref{4}) we have%
\begin{equation*}
f_{1}(x,u_{k},v_{k})(u_{k}-\hat u)\leq k_{1}d(x)^{\alpha }(u_{k}-\hat u)\leq
2k_{1}d(x)^{\alpha }\left\Vert \overline{u}\right\Vert _{\infty }
\end{equation*}%
and%
\begin{equation*}
f_{2}(x,u_{k},v_{k})(v_{k}-\hat v)\leq k_{2}d(x)^{\beta }(v_{k}-\hat v)\leq
k_{2}d(x)^{\beta }\left\Vert \overline{v}\right\Vert _{\infty }.
\end{equation*}%
Thank's to \cite[Lemma]{LM}, we deduce that $f_{1}(x,u_{k},v_{k})(u_{k}-\hat
u)$ and $f_{2}(x,u_{k},v_{k})(v_{k}-\hat v)$ are dominated by $L^{1}(\Omega
) $ functions and using the dominated convergence theorem, we obtain%
\begin{equation*}
\begin{array}{c}
\underset{k\rightarrow \infty }{\lim }\langle -\Delta _{p}u_{k},u_{k}-\hat
u\rangle =\underset{k\rightarrow \infty }{\lim }\langle -\Delta
_{q}v_{k},v_{k}-\hat v\rangle =0.%
\end{array}%
\end{equation*}%
Then the $S_{+}$-property of $-\Delta _{p}$ and $-\Delta _{q}$ on $%
W_{0}^{1,p}(\Omega )$ and $W_{0}^{1,q}(\Omega ),$ respectively, guarantees
that 
\begin{equation*}
\begin{array}{c}
(u_{k},v_{k})\longrightarrow (\hat u,\hat v)\text{ in }W_{0}^{1,p}(\Omega
)\times W_{0}^{1,q}(\Omega )%
\end{array}%
\end{equation*}%
and therefore $(\hat u,\hat v)$ is a positive solution of problem (\ref{p*}%
). Consequently, $(\hat u,\hat v)=\inf C$ belongs to $S$. Then Zorn's Lemma
can be applied which provides a minimal element $(u^*,v^*)$ of $S$.
Furthermore, since $S\subset [\underline u,\overline u]\times[\underline v,
\overline v]$, (\ref{h}) enables us to apply the regularity theory (see \cite%
{Hai}) to infer that $(u^*,v^*)\in C_{0}^{1,\gamma }(\overline{\Omega }%
)\times C_{0}^{1,\gamma }(\overline{ \Omega })$ for some $\gamma \in (0,1)$.

The proof is completed by showing that $(u^{\ast },v^{\ast })$ is the
smallest solution of (\ref{p*}) in $S$. To this end, let $(u,v)\in S$.
Bearing in mind that $S$ is downward directed, there is $(\mathring{u},%
\mathring{v})\in S$ with $\mathring{u}\leq u^{\ast },$ $\mathring{v}\leq
v^{\ast }$ and $\mathring{u}\leq u,$ $\mathring{v}\leq v$. Since $(u^{\ast
},v^{\ast })$ is a minimal element of $S$, it turns out that $(u^{\ast
},v^{\ast })=(\mathring{u},\mathring{v})\leq (u,v)$. The same reasoning can
be used to prove the existence of a biggest solution $(u^{+},v^{+})$ in $%
C_{0}^{1,\gamma }(\overline{\Omega })\times C_{0}^{1,\gamma }(\overline{%
\Omega }),$ for certain $\gamma \in (0,1)$, within $[\underline{u},\overline{%
u}]\times \lbrack \underline{v},\overline{v}]$. This completes the proof.
\end{proof}

\section{Auxiliary system}

\label{S3}

For every $z_{1},z_{2}\in C_{0}^{1}(\Omega ),$ let us state the auxiliary
problem%
\begin{equation}
\left\{ 
\begin{array}{ll}
-\Delta _{p}u=f(x,u,v,\nabla z_{1},\nabla z_{2}) & \text{in }\Omega , \\ 
-\Delta _{q}v=g(x,u,v,\nabla z_{1},\nabla z_{2}) & \text{in }\Omega , \\ 
u,v>0 & \text{in }\Omega , \\ 
u,v=0 & \text{on }\partial \Omega .%
\end{array}%
\right.  \tag{$P_{(z_{1},z_{2})}$}  \label{pz}
\end{equation}

With the aim of finding pairs of sub-supersolutions of problem (\ref{pz}),
let us define $\xi _{1}$ and $\xi _{2}$ in $C_{0}^{1,\beta }(\overline{%
\Omega }),$ $\beta \in (0,1)$, as the unique solutions of the problems 
\begin{equation}
\left\{ 
\begin{array}{ll}
-\Delta _{p}\xi _{1}=1 & \text{in }\Omega , \\ 
\xi _{1}=0 & \text{on }\partial \Omega%
\end{array}%
\right. \text{ \ and \ }\left\{ 
\begin{array}{ll}
-\Delta _{q}\xi _{2}=1 & \text{in }\Omega , \\ 
\xi _{2}=0 & \text{on }\partial \Omega ,%
\end{array}%
\right.  \label{20}
\end{equation}%
respectively, which are known to satisfy 
\begin{equation}
c_{0}d(x)\leq \xi _{1}(x)\leq c_{1}d(x)\text{ \ and \ }c_{0}^{\prime
}d(x)\leq \xi _{2}(x)\leq c_{1}^{\prime }d(x)\text{ \ in }\Omega ,
\label{21}
\end{equation}%
with constants $c_{i},c_{i}^{\prime }>0$ (see \cite{CMM}).

Set%
\begin{equation}
(\overline{u},\overline{v})=C(\xi _{1},\xi _{2})\text{ \ \ and \ \ }\left( 
\underline{u},\underline{v}\right) =C^{-1}\left( \phi _{1,p},\phi
_{1,q}\right) ,  \label{11}
\end{equation}%
where $C>1$ is a constant that will be fixed large enough and denote by 
\begin{equation}
\begin{array}{c}
M=\max \{\max_\Omega \phi _{1,p},\max_\Omega \phi _{1,q}\}.%
\end{array}
\label{15}
\end{equation}%
Obviously, as a consequence of the maximum principle, we have 
\begin{equation}  \label{ordered}
\left(\overline{u},\overline{v}\right) \geq \left(\underline{u},\underline{v}%
\right)\ \hbox{in}\ \overline{\Omega }\ \hbox{for}\ C>1\ \hbox{large}.
\end{equation}

Recall from \cite[Lemma 1]{BE} that if $h_{1},h_{2}\in L^{\infty }(\Omega )$
and $u\in W_{0}^{1,p}(\Omega ),$ $v\in W_{0}^{1,q}(\Omega )$ are the weak
solutions of problems 
\begin{equation*}
\left\{ 
\begin{array}{ll}
-\Delta _{p}u=h_{1} & \text{ in }\Omega \\ 
u=0 & \text{ on }\partial \Omega%
\end{array}%
\right. ,\text{ \ }\left\{ 
\begin{array}{ll}
-\Delta _{q}v=h_{2} & \text{ in }\Omega \\ 
v=0 & \text{ on }\partial \Omega ,%
\end{array}%
\right.
\end{equation*}%
there exist positive constants $K_{p}=K_{p}(p,N,\Omega )$ and $%
K_{q}=K_{q}(q,N,\Omega )$ such that 
\begin{equation}
\Vert \nabla u\Vert _{\infty }\leq K_{p}\Vert h_{1}\Vert _{\infty }^{\frac{1%
}{p-1}}\text{ \ and \ }\Vert \nabla v\Vert _{\infty }\leq K_{q}\Vert
h_{2}\Vert _{\infty }^{\frac{1}{q-1}}.  \label{12}
\end{equation}%
Denote by 
\begin{equation*}
\begin{array}{c}
R_{1}=\max \{\left\Vert \xi _{1}\right\Vert _{C_{0}^{1,\beta }(\overline{%
\Omega })},K_{p}\}\text{ and }R_{2}=\max \{\left\Vert \xi _{2}\right\Vert
_{C_{0}^{1,\beta }(\overline{\Omega })},K_{q}\}.%
\end{array}%
\end{equation*}%
Using the functions in (\ref{11}), we introduce the sets%
\begin{equation*}
\begin{array}{c}
\mathcal{K}_{1}(C)=\left\{ y\in C_{0}^{1}(\overline{\Omega }):\underline{u}%
\leq y\leq \overline{u}\text{ in }\Omega ,\text{ }\left\Vert \nabla
y\right\Vert _{\infty }\leq CR_{1}\right\}%
\end{array}%
\end{equation*}%
and 
\begin{equation*}
\begin{array}{c}
\mathcal{K}_{2}(C)=\left\{ y\in C_{0}^{1}(\overline{\Omega }):\underline{v}%
\leq y\leq \overline{v}\text{ in }\Omega ,\text{\ }\left\Vert \nabla
y\right\Vert _{\infty }\leq CR_{2}\right\} ,%
\end{array}%
\end{equation*}%
which are closed, bounded and convex in $C_{0}^{1}(\overline{\Omega })$.

\begin{proposition}
\label{P1} Assume $\mathrm{H}(f)$ and $\mathrm{H}(g)$. Then, for $C>1$
sufficiently large and for every $(z_{1},z_{2})\in \mathcal{K}_{1}(C)\times 
\mathcal{K}_{2}(C)$, problem (\ref{pz}) has a smallest solution $(u^{\ast
},v^{\ast })_{(z_{1},z_{2})}$ in $C_{0}^{1,\gamma }(\overline{\Omega }%
)\times C_{0}^{1,\gamma }(\overline{\Omega })$, for certain $\gamma \in
(0,1) $, within $[\underline{u},\overline{u}]\times \lbrack \underline{v},%
\overline{v}]$.
\end{proposition}

\begin{proof}
The proof is related to Theorem \ref{T4}. First, let us prove that

\begin{itemize}
\item[\protect\underline{Claim}:] \textit{For every $(z_{1},z_{2})\in 
\mathcal{K}_{1}(C)\times \mathcal{K}_{2}(C)$, $(\underline{u},\underline{v})$%
, $(\bar{u},\bar{v})$ form a pair of sub-supersolution for (\ref{pz})
provided that $C$ is large enough.}
\end{itemize}

\noindent Using $\mathrm{H}(f)$, $\mathrm{H}(g)$, (\ref{11}), (\ref{15}) and
(\ref{10}) we get%
\begin{equation*}
\begin{array}{l}
\underline{u}^{-\alpha _{1}}\underline{v}^{-\beta _{1}}(-\Delta _{p}%
\underline{u})=C^{-(p-1-\alpha _{1}-\beta _{1})}\lambda _{1,p}\phi
_{1,p}^{p-1-\alpha _{1}}\phi _{1,q}^{-\beta _{1}} \\ 
\leq C^{-(p-1-\alpha _{1}-\beta _{1})}\lambda _{1,p}l_{1}^{-\beta _{1}}\phi
_{1,p}^{p-1-\alpha _{1}-\beta _{1}} \\ 
\leq C^{-(p-1-\alpha _{1}-\beta _{1})}\lambda _{1,p}l_{1}^{-\beta
_{1}}M^{p-1-\alpha _{1}-\beta _{1}}\leq m_{1}\text{ in }\Omega%
\end{array}%
\end{equation*}%
and%
\begin{equation*}
\begin{array}{l}
\underline{u}^{-\alpha _{2}}\underline{v}^{-\beta _{2}}(-\Delta _{q}%
\underline{v})=C^{-(q-1-\alpha _{2}-\beta _{2})}\lambda _{1,q}\phi
_{1,q}^{q-1-\beta _{2}}\phi _{1,p}^{-\alpha _{2}} \\ 
\leq C^{-(q-1-\alpha _{2}-\beta _{2})}\lambda _{1,q}l_{2}^{\alpha _{2}}\phi
_{1,q}^{q-1-\alpha _{2}-\beta _{2}} \\ 
C^{-(q-1-\alpha _{2}-\beta _{2})}\lambda _{1,q}l_{2}^{\alpha
_{2}}M^{q-1-\alpha _{2}-\beta _{2}}\leq m_{2}\text{ in }\Omega ,%
\end{array}%
\end{equation*}%
provided that $C>1$ is large enough (such that (\ref{ordered}) holds too).
Then, it is readily seen from $\mathrm{H}(f)$ and $\mathrm{H}(g)$ that%
\begin{eqnarray}  \label{3.6}
\int_{\Omega }\left\vert \nabla \underline{u}\right\vert ^{p-2}\nabla 
\underline{u}\nabla \varphi & \leq & m_{1}\int_{\Omega }\underline{u}%
^{\alpha _{1}}\underline{v}^{\beta _{1}}\varphi \leq m_{1}\int_{\Omega }%
\underline{u}^{\alpha _{1}}w_{2}^{\beta _{1}}\varphi \\
& \leq & \int_{\Omega }f(x,\underline{u},w_{2},\nabla z_{1},\nabla
z_{2})\varphi  \notag
\end{eqnarray}%
%
%
%
and%
\begin{eqnarray}  \label{3.7}
\int_{\Omega }\left\vert \nabla \underline{v}\right\vert ^{q-2}\nabla 
\underline{v}\nabla \psi & \leq & m_{2}\int_{\Omega }\underline{u}^{\alpha
_{2}}\underline{v}^{\beta _{2}}\psi \leq m_{2}\int_{\Omega }w_{1}^{\alpha
_{2}}\underline{v}^{\beta _{2}}\psi \\
& \leq &\int_{\Omega }g(x,w_{1},\underline{v},\nabla z_{1},\nabla z_{2})\psi
,  \notag
\end{eqnarray}%
%
%
%
for all $\left( \varphi ,\psi \right) \in W_{0}^{1,p}\left( \Omega \right)
\times W_{0}^{1,q}\left( \Omega \right) $ with $\varphi ,\psi \geq 0$ in $%
\Omega ,$ for all $\left( w_{1},w_{2}\right) \in W^{1,p}(\Omega )\times
W^{1,q}(\Omega )$ satisfying $\underline{u}\leq w_{1}\leq \overline{u}$ and $%
\underline{v}\leq w_{2}\leq \overline{v}$ a.e. in $\Omega $, and for $%
(z_{1},z_{2})\in \mathcal{K}_{1}(C)\times \mathcal{K}_{2}(C)$.

Now, taking into account (\ref{20}), (\ref{21}), (\ref{15}), $\mathrm{H}(f)$
and $\mathrm{H}(g)$, we derive the estimates 
\begin{equation*}
\begin{array}{l}
\overline{u}^{-\alpha _{1}}\overline{v}^{-\beta _{1}}(-\Delta _{p}\overline{u%
}-(CR_{1})^{\gamma _{1}}-(CR_{2})^{\theta _{1}}) \\ 
=C^{-\alpha _{1}-\beta _{1}}\xi _{1}^{-\alpha _{1}}\xi _{2}^{-\beta
_{1}}(C^{p-1}-(CR_{1})^{\gamma _{1}}-(CR_{2})^{\theta _{1}}) \\ 
=C^{p-1-\alpha _{1}-\beta _{1}}\xi _{1}^{-\alpha _{1}}\xi _{2}^{-\beta _{1}} 
\left[ 1-C^{-(p-1)}(C^{\gamma _{1}}R_{1}^{\gamma _{1}}+C^{\theta
_{1}}R_{2}^{\theta _{1}}\right] \\ 
\geq C^{p-1-\alpha _{1}-\beta _{1}}(c_{0}d(x))^{-\alpha _{1}}(c_{1}^{\prime
}d(x))^{-\beta _{1}}\left[ 1-(C^{\gamma _{1}-(p-1)}R_{1}^{\gamma
_{1}}+C^{\theta _{1}-(p-1)}R_{2}^{\theta _{1}})\right] \\ 
\geq C^{p-1-\alpha _{1}-\beta _{1}}c_{0}^{-\alpha _{1}}(c_{1}^{\prime
})^{-\beta _{1}}d(x)^{-(\alpha _{1}+\beta _{1})}\left[ 1-(C^{\gamma
_{1}-(p-1)}R_{1}^{\gamma _{1}}+C^{\theta _{1}-(p-1)}R_{2}^{\theta _{1}})%
\right] \\ 
\geq M_{1}\text{ in }\ \Omega%
\end{array}%
\end{equation*}%
and%
\begin{equation*}
\begin{array}{l}
\overline{u}^{-\alpha _{2}}\overline{v}^{-\beta _{2}}(-\Delta _{q}\overline{v%
}-(CR_{1})^{\gamma _{2}}-(CR_{2})^{\theta _{2}}) \\ 
{\geq }C^{q-1-\alpha _{2}-\beta _{2}}c_{1}^{-\alpha _{2}}(c_{0}^{\prime
})^{-\beta _{2}}d(x)^{-(\alpha _{2}+\beta _{2})}\left[ 1-(C^{\gamma
_{2}-(q-1)}R_{1}^{\gamma _{2}}+C^{\theta _{2}-(q-1)}R_{2}^{\theta _{2}})%
\right] \\ 
\geq M_{2}\text{ in }\ \Omega,%
\end{array}%
\end{equation*}%
provided that $C>1$ is sufficiently large. Consequently, it turns out that%
\begin{eqnarray}  \label{33}
\int_{\Omega }\left\vert \nabla \overline{u}\right\vert ^{p-2}\nabla 
\overline{u}\nabla \varphi \text{ }dx & \geq & \int_{\Omega }(M_{1}\overline{%
u}^{\alpha _{1}}\overline{v}^{\beta _{1}}+|\nabla z_{1}|^{\gamma
_{1}}+|\nabla z_{2}|^{\theta _{1}})\varphi \text{ }dx \\
& \geq & \int_{\Omega }(M_{1}\overline{u}^{\alpha _{1}}w_{2}^{\beta
_{1}}+|\nabla z_{1}|^{\gamma _{1}}+|\nabla z_{2}|^{\theta _{1}})\varphi 
\text{ }dx  \notag \\
& \geq & \int_{\Omega }f(x,\overline{u},w_{2},\nabla z_{1},\nabla
z_{2})\varphi \text{ }dx  \notag
\end{eqnarray}
and 
\begin{eqnarray}  \label{34}
\int_{\Omega }\left\vert \nabla \overline{v}\right\vert ^{q-2}\nabla 
\overline{v}\nabla \psi \text{ }dx & \geq & \int_{\Omega }M_{2}(\overline{u}%
^{\alpha _{2}}\overline{v}^{\beta _{2}}+|\nabla z_{1}|^{\gamma _{2}}+|\nabla
z_{2}|^{\theta _{2}})\psi \text{ }dx \\
& \geq & \int_{\Omega }M_{2}(w_{1}^{\alpha _{2}}\overline{v}^{\beta
_{2}}+|\nabla z_{1}|^{\gamma _{2}}+|\nabla z_{2}|^{\theta _{2}})\psi \text{ }%
dx  \notag \\
& \geq & \int_{\Omega }g(x,w_{1},\overline{v},\nabla z_{1},\nabla z_{2})\psi 
\text{ }dx  \notag
\end{eqnarray}
for all $\left( \varphi ,\psi \right) \in W_{0}^{1,p}\left( \Omega \right)
\times W_{0}^{1,q}\left( \Omega \right) $ with $\varphi ,\psi \geq 0$ in $%
\Omega $, for all $\left( w_{1},w_{2}\right) \in W^{1,p}(\Omega )\times
W^{1,q}(\Omega )$ satisfying $\underline{u}\leq w_{1}\leq \overline{u}$ and $%
\underline{v}\leq w_{2}\leq \overline{v}$ a.e. in $\Omega $, and for $%
(z_{1},z_{2})\in \mathcal{K}_{1}(C)\times \mathcal{K}_{2}(C)$.\newline
Putting together (\ref{3.6}), (\ref{3.7}), (\ref{33}) and (\ref{34}) we get
the \underline{Claim}.

Furthermore, for every $(z_{1},z_{2})\in \mathcal{K}_{1}(C)\times \mathcal{K}%
_{2}(C)$ and for every $(u,v)\in \lbrack \underline{u},\overline{u}]\times
\lbrack \underline{v},\overline{v}]$, from $\mathrm{H}(f)$, $\mathrm{H}(g)$,
(\ref{11}), (\ref{15}) and (\ref{9}) we have the estimates%
\begin{equation}
\begin{array}{l}
f(x,u,v,\nabla z_{1},\nabla z_{2})\leq M_{1}u^{\alpha _{1}}v^{\beta
_{1}}+|\nabla z_{1}|^{\gamma _{1}}+|\nabla z_{2}|^{\theta _{1}} \\ 
\leq M_{1}\underline{u}^{\alpha _{1}}\overline{v}^{\beta
_{1}}+(CR_{1})^{\gamma _{1}}+(CR_{2})^{\theta _{1}} \\ 
\leq M_{1}C^{-\alpha _{1}+\beta _{1}}\phi _{1,p}^{\alpha _{1}}\left\Vert \xi
_{2}\right\Vert _{\infty }^{\beta _{1}}+(CR_{1})^{\gamma
_{1}}+(CR_{2})^{\theta _{1}} \\ 
\leq \phi _{1,p}^{\alpha _{1}}\left( M_{1}C^{-\alpha _{1}+\beta
_{1}}\left\Vert \xi _{2}\right\Vert _{\infty }^{\beta _{1}}+M^{-\alpha
_{1}}(CR_{1})^{\gamma _{1}}+M^{-\alpha _{1}}(CR_{2})^{\theta _{1}}\right) \\ 
\leq C_{1}d(x)^{\alpha _{1}}\text{ \ in }\Omega%
\end{array}
\label{5}
\end{equation}%
and%
\begin{equation}
\begin{array}{l}
g(x,u,v,\nabla z_{1},\nabla z_{2})\leq M_{2}y_{1}^{\alpha _{2}}y_{2}^{\beta
_{2}}+|\nabla z_{1}|^{\gamma _{2}}+|\nabla z_{2}|^{\theta _{2}} \\ 
\leq M_{2}\overline{u}^{\alpha _{2}}\underline{v}^{\beta
_{2}}+(CR_{1})^{\gamma _{2}}+(CR_{2})^{\theta _{2}} \\ 
\leq M_{2}C^{\alpha _{2}-\beta _{2}}\left\Vert \xi _{1}\right\Vert _{\infty
}^{\alpha _{2}}\phi _{1,q}^{\beta _{2}}+(CR_{1})^{\gamma
_{2}}+(CR_{2})^{\theta _{2}} \\ 
\leq \phi _{1,q}^{\beta _{2}}\left( M_{2}C^{\alpha _{2}+\beta
_{2}}\left\Vert \xi _{1}\right\Vert _{\infty }^{\alpha _{2}}+M^{-\beta
_{2}}(CR_{1})^{\gamma _{2}}+M^{-\beta _{2}}(CR_{2})^{\theta _{2}}\right) \\ 
\leq C_{2}d(x)^{\beta _{2}}\text{ \ in }\Omega ,%
\end{array}
\label{6}
\end{equation}%
for some positive constants $C_{1}$ and $C_{2}$ independent from $u$, $v$, $%
z_1$ and $z_2$. Then, owing to Theorem \ref{T4}, it follows that if $C>1$ is
large enough (according to the \underline{Claim}), for every $%
(z_{1},z_{2})\in \mathcal{K}_{1}(C)\times \mathcal{K}_{2}(C)$ problem (\ref%
{pz}) has a smallest solution $(u^{\ast },v^{\ast })_{(z_{1},z_{2})}\in
C_{0}^{1,\gamma }(\overline{\Omega })\times C_{0}^{1,\gamma }(\overline{%
\Omega })$ for certain $\gamma \in (0,1)$, within $[\underline{u},\overline{u%
}]\times \lbrack \underline{v},\overline{v}]$. This complete the proof.
\end{proof}

\begin{remark}
\label{Clarge} \textrm{We wish explicitly to point out that from the proof
of Proposition \ref{P1} one can derive an estimate of the largeness of $C>1$%
. In particular, the choice of $C$, that first of all is related to (\ref%
{ordered}), is crucial for verifying the \underline{Claim} and, as a
consequence, that for every $(z_1,z_2)\in \mathcal{K}_{1}(C)\times \mathcal{K%
}_{2}(C)$ the set of the solutions of problem (\ref{pz}) is nonempty.}
\end{remark}

\bigskip

In what follows, $C>1$ will be assumed large enough such that for any $%
(z_{1},z_{2})\in \mathcal{K}_{1}(C)\times \mathcal{K} _{2}(C)$, put 
\begin{equation*}
S_{(z_{1},z_{2})}=\{(u,v)\in[ \underline{u},\overline{u}]\times[ \underline{v%
}, \overline{v}]:\ (u,v)\in C_{0}^{1}(\overline{\Omega})\times C_{0}^{1}(%
\overline{\Omega})\ \hbox{solves } (P_{(z_1,z_2)}) \},
\end{equation*}
then $S_{(z_1,z_2)}\neq\emptyset$.

\begin{lemma}
\label{L4}Assume $\mathrm{H}(f)$ and $\mathrm{H}(g)$ hold and let $C>1$ be
large enough. If $\{(z_{1,n},z_{2,n})\}$ is a sequence in $\mathcal{K}%
_{1}(C)\times \mathcal{K}_{2}(C)$ such that $(z_{1,n},z_{2,n})\rightarrow
(z_{1},z_{2})$ in $\mathcal{K}_{1}(C)\times \mathcal{K}_{2}(C)$, then for
any $(\breve{u},\breve{v})\in S_{(z_{1},z_{2})}$, there exists $(\breve{u}%
_{n},\breve{v}_{n})\in S_{(z_{1,n},z_{2,n})}$ such that $(\breve{u}_{n},%
\breve{v}_{n})\rightarrow (\breve{u},\breve{v})$ in $C_{0}^{1}(\overline{%
\Omega })\times C_{0}^{1}(\overline{\Omega })$.
\end{lemma}

\begin{proof}
Fix $C>1$ large enough, put 
\begin{equation*}
\hat{\rho}=\max \left\{ -\frac{\alpha _{1}M_{1}}{p-1}(C^{-1}l)^{\alpha
_{1}-p+1}\hat{l}^{\beta _{1}}C^{-\beta _{1}},-\frac{\beta _{2}M_{2}}{q-1}%
(C^{-1}l)^{\beta _{2}-q+1}\hat{l}^{\alpha _{2}}C^{-\alpha _{2}}\right\} ,
\end{equation*}%
and observe that 
\begin{equation}
\left\{ 
\begin{array}{c}
\alpha _{1}\mu _{1}t^{\alpha _{1}-1}\underline{v}(x)^{\beta _{1}}+\hat{\rho}%
(p-1)d(x)^{\alpha _{1}+\beta _{1}-(p-1)}t^{p-2}\geq 0 \\ 
\beta _{2}\mu _{2}\underline{u}(x)^{\alpha _{2}}t^{\beta _{2}-1}+\hat{\rho}%
(q-1)d(x)^{\alpha _{2}+\beta _{2}-(q-1)}t^{q-2}\geq 0,%
\end{array}%
\right.  \label{17}
\end{equation}%
uniformly in $x\in \Omega $, for all $t\geq \min \{\underline{u}(x),%
\underline{v}(x)\}$, where $\mu _{1}\in \{m_{1},\,M_{1}\}$ and $\mu _{2}\in
\{m_{2},\,M_{2}\}$.\newline
Indeed, from (\ref{9}) and bearing in mind that $\alpha _{1}-p+1<0<\beta
_{1} $ one has that 
\begin{equation*}
\left( \frac{\underline{v}(x)}{d(x)}\right) ^{\beta _{1}}\leq C^{-\beta _{1}}%
\hat{l}_{1}^{\beta },\quad \max \left\{ \left( \frac{\underline{u}(x)}{d(x)}%
\right) ^{\alpha _{1}-p+1},\left( \frac{\underline{v}(x)}{d(x)}\right)
^{\alpha _{1}-p+1}\right\} \leq (C^{-1}l)^{\alpha _{1}-p+1}
\end{equation*}%
for all $x\in \Omega $. Hence, 
\begin{eqnarray*}
\hat{\rho} &\geq &-\frac{\alpha _{1}M_{1}}{p-1}\left( \frac{t}{d(x)}\right)
^{\alpha _{1}-p+1}\left( \frac{\underline{v}(x)}{d(x)}\right) ^{\beta _{1}}
\\
&=&\frac{-\alpha _{1}M_{1}t^{\alpha _{1}-1}\underline{v}(x)^{\beta _{1}}}{%
(p-1)d(x)^{\alpha _{1}+\beta _{1}-(p-1)}t^{p-2}} \\
&\geq &\frac{-\alpha _{1}m_{1}t^{\alpha _{1}-1}\underline{v}(x)^{\beta _{1}}%
}{(p-1)d(x)^{\alpha _{1}+\beta _{1}-(p-1)}t^{p-2}}.
\end{eqnarray*}%
for all $t\geq \min \{\underline{u}(x),\underline{v}(x)\}$ and uniformly in $%
\Omega $, so that the first inequality in (\ref{17}) holds. The second
inequality can be verified arguing in analogy.\newline
Here, condition (\ref{17}) guaranties that for all $(u,v)\in \mathcal{K}%
_{1}(C)\times \mathcal{K}_{2}(C)$ the functions%
\begin{equation}
\mu _{1}t^{\alpha _{1}}v(x)^{\beta _{1}}+\hat{\rho}d(x)^{\alpha _{1}+\beta
_{1}-(p-1)}t^{p-1},\quad \mu _{2}u(x)^{\alpha _{2}}t^{\beta _{2}}+\hat{\rho}%
d(x)^{\alpha _{2}+\beta _{2}-(q-1)}t^{q-1}  \label{monotone}
\end{equation}%
are monotone with respect to $t\geq \min \{\underline{u}(x),\underline{v}%
(x)\}$.\newline
Let now $\hat{f}$, $\hat{g}$ be the functions defined by 
\begin{equation*}
\hat{f}(x,s_{1},s_{2},\xi _{1},\xi _{2})=f(x,s_{1},s_{2},\xi _{1},\xi _{2})+%
\hat{\rho}d(x)^{\alpha _{1}+\beta _{1}-(p-1)}s_{1}^{p-1}
\end{equation*}%
\begin{equation*}
\hat{g}(x,s_{1},s_{2},\xi _{1},\xi _{2})=g(x,s_{1},s_{2},\xi _{1},\xi _{2})+%
\hat{\rho}d(x)^{\alpha _{2}+\beta _{2}-(q-1)}s_{2}^{q-1}
\end{equation*}%
for $(x,s_{1},s_{2},\xi _{1},\xi _{2})\in \Omega \times (0,+\infty )\times
(0,+\infty )\times \mathbb{R}^{2N}$.\newline
Arguing as in (\ref{5}) and (\ref{6}), bearing in mind (\ref{21}) and (\ref%
{11}), there exist two positive constants $\hat{C}_{1}$ and $\hat{C}_{2}$
such that 
\begin{equation}
\begin{array}{c}
\hat{f}(x,u,v,\nabla y_{1},\nabla y_{2})\leq \hat{C}_{1}d(x)^{\alpha
_{1}},\quad \hat{g}(x,u,v,\nabla y_{1},\nabla y_{2})\leq \hat{C}%
_{2}d(x)^{\beta _{2}}%
\end{array}
\label{14}
\end{equation}%
a.e. in $\Omega $, for every $(u,v)\in \lbrack \underline{u},\overline{u}%
]\times \lbrack \underline{v},\overline{v}]$ and every $(y_{1},y_{2})\in 
\mathcal{K}_{1}(C)\times \mathcal{K}_{2}(C)$.

Let us consider now the following differential operators $%
L_{p}:W_{0}^{1,p}(\Omega )\rightarrow W^{-1,p^{\prime }}(\Omega )$, $%
L_{q}:W_{0}^{1,q}(\Omega )\rightarrow W^{-1,q^{\prime }}(\Omega )$ defined
by 
\begin{equation*}
L_{p}(u)=-\Delta _{p}u+\hat{\rho}d(x)^{\alpha _{1}+\beta
_{1}-(p-1)}|u|^{p-2}u
\end{equation*}%
for all $u\in W_{0}^{1,p}(\Omega )$, and 
\begin{equation*}
L_{q}(u)=-\Delta _{q}u+\hat{\rho}d(x)^{\alpha _{2}+\beta
_{2}-(q-1)}|u|^{q-2}u
\end{equation*}%
for all $u\in W_{0}^{1,q}(\Omega )$. Observing that $p^{\prime }(\alpha
_{1}+\beta _{1})-p>-p$ and $q^{\prime }(\alpha _{2}+\beta _{2})-q>-q$, one
can apply \cite[Theorem 19.8]{OpiKuf} ($\partial \Omega $ is assumed to be
smooth enough) in order to obtain 
\begin{equation*}
d(x)^{\alpha _{1}+\beta _{1}-(p-1)}|u|^{p-2}u\in L^{p^{\prime }}(\Omega
),\quad d(x)^{\alpha _{2}+\beta _{2}-(q-1)}|u|^{q-2}u\in L^{q^{\prime
}}(\Omega ),
\end{equation*}%
namely $L_{p}$ and $L_{q}$ are well defined.\newline
A direct computation shows that $L_{p}$ and $L_{q}$ are demicontinuous,
coercive and strictly monotone. Hence, in view of (\ref{14}), one can apply
the Minty-Browder theorem and conclude that for every $(u,v)\in \lbrack 
\underline{u},\overline{u}]\times \lbrack \underline{v},\overline{v}]$, for
every $(y_{1},y_{2})\in \mathcal{K}_{1}(C)\times \mathcal{K}_{2}(C)$ the
problem 
\begin{equation}
\left\{ 
\begin{array}{ll}
L_{p}(w_{1})=\hat{f}(x,u,v,\nabla y_{1},\nabla y_{2}) & \text{in }\Omega ,
\\ 
L_{q}(w_{2})=\hat{g}(x,u,v,\nabla y_{1},\nabla y_{2}) & \text{in }\Omega ,
\\ 
w_{1},w_{2}=0 & \text{on }\partial \Omega ,%
\end{array}%
\right.  \tag{$P_{(u,v,y_{1},y_{2})}$}  \label{ppp}
\end{equation}%
admits a unique solution.

At this point fix $(z_1,z_2)\in \mathcal{K}_{1}(C)\times \mathcal{K}_{2}(C)$%
, $(\breve u,\breve v)\in S_{(z_1,z_2)}$ and let $\{(z_{1,n},z_{2,n})\}$ be
a sequence in $\mathcal{K}_{1}(C)\times \mathcal{K}_{2}(C)$ such that $%
(z_{1,n},z_{2,n})\to (z_1,z_2)$.\newline
Obviously, $(\breve u,\breve v)\in S_{(z_1,z_2)}$ implies that 
\begin{equation}  \label{unique}
(\breve u,\breve v)\ \hbox{is the unique solution of } (P_{(\breve u,\breve
v,z_1,z_2)}).
\end{equation}
Fix $n\in 
\mathbb{N}
$ and let $(w_{1,n}^{0},w_{2,n}^{0})$ be the unique solution of the problem $%
(P_{(\breve u,\breve v,z_{1,n},z_{2,n})})$.

By $\mathrm{H}(f)$ and $\mathrm{H}(g)$, since $(\breve{u},\breve{v})\in
[\underline u,\overline u]\times[\underline v,\overline v]$, using the
monotonicity of the functions introduced in (\ref{monotone}) and the
computations pointed out in (\ref{3.6}) and (\ref{3.7}), it follows that%
\begin{eqnarray*}
L_{p}(w_{1,n}^{0})& = & \hat f(x,\breve u, \breve v,\nabla z_{1,n},\nabla
z_{2,n})= f(x,\breve{u},\breve{v},\nabla z_{1,n},\nabla z_{2,n})+\hat\rho
d(x)^{\alpha _{1}+\beta _{1}-(p-1)}\breve{u}^{p-1} \\
& \geq & m_{1}\breve{u}^{\alpha _{1}}\breve{v}^{\beta _{1}}+\hat\rho
d(x)^{\alpha _{1}+\beta _{1}-(p-1)}\breve{u}^{p-1} \\
& \geq & m_{1}\underline{u}^{\alpha _{1}}\breve{v}^{\beta _{1}}+\hat\rho
d(x)^{\alpha _{1}+\beta _{1}-(p-1)}\underline{u}^{p-1} \\
& \geq & m_{1}\underline{u}^{\alpha _{1}}\underline{v}^{\beta _{1}}+\hat\rho
d(x)^{\alpha _{1}+\beta _{1}-(p-1)}\underline{u}^{p-1} \\
& \geq & -\Delta _{p}\underline{u}+\hat\rho d(x)^{\alpha _{1}+\beta
_{1}-(p-1)}\underline{u}^{p-1}=L_{p}(\underline{u})
\end{eqnarray*}%
and similarly, we obtain%
\begin{eqnarray*}
L_{q}(w_{2,n}^{0})& = & \hat g(x,\breve{u},\breve{v},\nabla z_{1,n},\nabla
z_{2,n}) \\
& \geq & -\Delta _{q}\underline{v}+\hat\rho C_{2}d(x)^{\beta _{2}-q+1}%
\underline{v}^{q-1}=L_{q}(\underline{v}).
\end{eqnarray*}%
The same reasoning can be exploited for assuring that 
\begin{equation*}
L_{p}(w_{1,n}^{0})\leq L_{p}(\overline{u})\text{ and \ }L_{p}(w_{1,n}^{0})%
\leq L_{p}(\overline{v}).
\end{equation*}%
Accordingly, the weak comparison principle in \cite{T} implies that $%
(w_{1,n}^{0},w_{2,n}^{0})\in \lbrack \underline{u},\overline{u}]\times
\lbrack \underline{v},\overline{v}]$. Furthermore, from (\ref{14}), by the
regularity theory (see \cite[Lemma 3.1]{Hai}), it follows $%
(w_{1,n}^{0},w_{2,n}^{0})\in C_{0}^{1,\gamma }(\overline{\Omega })\times
C_{0}^{1,\gamma }(\overline{\Omega })$ for some $\gamma\in (0,1)$, and, in
particular, $\{(w_{1,n},w_{2,n})\}$ is bounded in $C_{0}^{1,\gamma }(%
\overline{\Omega })\times C_{0}^{1,\gamma }(\overline{\Omega })$. Then,
since $C_{0}^{1,\gamma }(\overline{\Omega })\subset C_{0}^{1}(\overline{%
\Omega })$ is compact, there exist a subsequence, denoted by the same
symbol, $\{(w_{1,n}^{0},w_{2,n}^{0})\}$ and $(\hat u,\hat v)$ such that 
\begin{equation}
(w_{1,n}^{0},w_{2,n}^{0})\rightarrow (\hat{u},\hat{v})\, \text{\ in }%
C_{0}^{1}(\overline{\Omega })\times C_{0}^{1}(\overline{\Omega }).
\label{32}
\end{equation}%
Hence, passing to the limit in $(P_{(\breve u,\breve v,z_{1,n},z_{2,n})})$),
one has that $(\hat u,\hat v)$ is a solution of the problem $(P_{\breve
u,\breve v,z_1,z_2})$. Namely, in view of (\ref{unique}), $(\hat{u},\hat{v}%
)=(\breve{u},\breve{v})$ and by the strong convergence (\ref{32}) we infer
that 
\begin{equation*}
\lim_{n\rightarrow \infty }(w_{1,n}^{0},w_{2,n}^{0})=(\breve{u},\breve{v})%
\text{ \ in }C_{0}^{1}(\overline{\Omega })\times C_{0}^{1}(\overline{\Omega }%
).
\end{equation*}

Now, let $(w_{1,n}^{1},w_{2,n}^{1})$ be the unique solution of the problem $%
(P_{w_{1,n}^0,w_{2,n}^0,z_{1,n},z_{2,n}})$. Following the same argument as
before we obtain%
\begin{equation*}
(w_{1,n}^{1},w_{2,n}^{1})\in \lbrack \underline{u},\overline{u}]\times
\lbrack \underline{v},\overline{v}]
\end{equation*}%
and 
\begin{equation*}
\lim_{n\rightarrow \infty }(w_{1,n}^{1},w_{2,n}^{1})=(\breve{u},\breve{v})%
\text{ \ in }C_{0}^{1}(\overline{\Omega })\times C_{0}^{1}(\overline{\Omega }%
).
\end{equation*}%
Inductively, for each $n\in 
\mathbb{N}
$, we construct the sequences $\{(w_{1,n}^{k},w_{2,n}^{k})\}_{k}$ in $%
C_{0}^{1}(\overline{\Omega })\times C_{0}^{1}(\overline{\Omega })$ as a
unique solution of%
\begin{equation}
\left\{ 
\begin{array}{ll}
-\Delta _{p}u=\hat{f}(x,w_{1,n}^{k-1},w_{2,n}^{k-1},\nabla z_{1,n},\nabla
z_{2,n}) & \text{in }\Omega , \\ 
-\Delta _{q}v=\hat{g}(x,w_{1,n}^{k-1},w_{2,n}^{k-1},\nabla z_{1,n},\nabla
z_{2,n}) & \text{in }\Omega , \\ 
u,v=0 & \text{on }\partial \Omega ,%
\end{array}%
\right.  \label{26}
\end{equation}%
such that for each $k\in 
\mathbb{N}
,$ we have%
\begin{equation*}
(w_{1,n}^{k},w_{2,n}^{k})\in \lbrack \underline{u},\overline{u}]\times
\lbrack \underline{v},\overline{v}]
\end{equation*}%
and%
\begin{equation*}
\lim_{n\rightarrow \infty }(w_{1,n}^{k},w_{2,n}^{k})=(\breve{u},\breve{v})%
\text{ \ in }C_{0}^{1}(\overline{\Omega })\times C_{0}^{1}(\overline{\Omega }%
).
\end{equation*}%
The task is now to show that $\{(w_{1,n}^{k},w_{2,n}^{k})\}_{n,k}$ is
relatively compact in $C_{0}^{1}(\overline\Omega )\times
C_{0}^{1}(\overline\Omega )$. Indeed, bearing in mind that $%
(w_{1,n}^{k-1},w_{2,n}^{k-1})\in \lbrack \underline{u},\overline{u}]\times
\lbrack \underline{v},\overline{v}]$ and $(z_{1,n},z_{2,n})\in \mathcal{K}%
_{1}(C)\times \mathcal{K}_{2}(C)$, on account of (\ref{14}), one gets%
\begin{equation*}
\hat f(x,w_{1,n}^k,w_{2,n}^k,z_{1,n},z_{2,n})\leq \hat C_1
d(x)^{\alpha_1},\quad \hat g(x,w_{1,n}^k,w_{2,n}^k,z_{1,n},z_{2,n})\leq \hat
C_2 d(x)^{\beta_2},
\end{equation*}%
for every $n,\, k\in\mathbb{N}$, with $\hat C_{1},\, \hat C_{2}>0$
independent from $n$ and $k$. Applying \cite[Lemma 3.1]{Hai}, $%
\{(w_{1,n}^{k},w_{2,n}^{k})\}_{n,k}$ is bounded in $C_0^{1,\gamma}(\overline%
\Omega)\times C_0^{1,\gamma}(\overline\Omega)$ and our task is achieved,
being $C_0^{1,\gamma}(\overline \Omega)$ compactly embedded in $%
C_0^1(\overline \Omega)$. Finally, the conclusion follows by proceeding
analogously to the proof of \cite[Lemma 2.5, page 535]{FMP}.
\end{proof}

\section{Proof of the main result}

\label{S4}

According to Proposition \ref{P1}, for $C>1$ large enough, for all $%
(z_{1},z_{2})\in \mathcal{K}_{1}(C)\times \mathcal{K}_{2}(C)$, there exists $%
(u^{\ast },v^{\ast })_{(z_{1},z_{2})}$ in $C_{0}^{1,\gamma }(\overline{%
\Omega })\times C_{0}^{1,\gamma }(\overline{\Omega })$ for certain $\gamma
\in (0,1)$ that is the smallest solution within $[\underline{u},\overline{u}%
]\times \lbrack \underline{v},\overline{v}]$ for system (\ref{pz}). Thus,
the operator 
\begin{equation*}
\begin{array}{lll}
\mathcal{T}: & \mathcal{K}_{1}(C)\times \mathcal{K}_{2}(C)\rightarrow & 
C_{0}^{1}(\overline{\Omega })\times C_{0}^{1}(\overline{\Omega }) \\ 
& \text{ }(z_{1},z_{2})\text{\ }\mapsto & \mathcal{T}(z_{1},z_{2})=(u^{\ast
},v^{\ast })_{(z_{1},z_{2})}%
\end{array}%
\end{equation*}%
is well defined and clearly the fixed points of the map $\mathcal{T}$ are
solutions of problem (\ref{p}).

\begin{lemma}
\label{L2} The map $\mathcal{T}$ is continuous and compact.
\end{lemma}

\begin{proof}
First, observe that $\mathcal{T}$ is compact, namely, taking in mind that $%
\mathcal{K}_{1}(C)\times \mathcal{K}_{2}(C)$ is bounded with respect to the $%
(C_0^1(\overline\Omega)\times C_0^1(\overline\Omega))$-topology, for any
sequence $\{(z_{1,n},z_{2,n})\}_{n}$ in $\mathcal{K}_{1}(C)\times \mathcal{K}%
_{2}(C)$ one has that $(u_{n}^{\ast },v_{n}^{\ast })=\mathcal{T}%
(z_{1,n},z_{2,n})$ is relatively compact in $C_{0}^{1}(\overline{\Omega }%
)\times C_{0}^{1}(\overline{\Omega }) $. This follows readily from (\ref{5})
and (\ref{6}), since $(\underline{u},\underline{v})\leq (u_{n}^{\ast
},v_{n}^{\ast })\leq (\overline{u},\overline{v})$ in $\Omega$, Indeed, as in
the proof of Lemma \ref{L4}, applying \cite[Lemma 3.1]{Hai} one has that $%
\{(u_n,v_n)\}$ is bounded in $C_0^{1,\gamma}(\overline\Omega)\times
C_0^{1,\gamma}(\overline\Omega)$ and we can conclude again invoking the
compactness of the embedding $C_0^{1,\gamma}(\overline
\Omega)\hookrightarrow C_0^{1}(\overline \Omega)$.

Let us show that $\mathcal{T}$ is continuous. Let $(z_{1,n},z_{2,n})%
\rightarrow (z_{1},z_{2})$ in $\mathcal{K}_{1}(C)\times \mathcal{K}_{2}(C)$
and put $(u_{n}^*,v_{n}^*)=\mathcal{T}(z_{1,n},z_{2,n})$. Then, we already
know that there exist a subsequence $\{(u_{n_{k}},v_{n_{k}})\}_{k}$ and an
element $(u_n^*,v_n^*)\in[\underline u,\overline u]\times[\underline
v,\overline v]$ such that 
\begin{equation}  \label{subsequence}
(u_{n_{k}},v_{n_{k}})\rightarrow (u,v)\text{ in }C_{0}^{1}(\overline{\Omega }%
)\times C_{0}^{1}(\overline{\Omega })\text{.}
\end{equation}
Passing to the limit in the equations 
\begin{equation*}
-\Delta _{p}u_{n_k}^*=f(x,u_{n_k}^*,v_{n_k}^*,\nabla z_{1,{n_k}},\nabla z_{2,%
{n_k}}),
\end{equation*}
\begin{equation*}
-\Delta _{q}v_{n_k}^*=g(x,u_{n_k}^*,v_{n_k}^*,\nabla z_{1,{n_k}},\nabla z_{2,%
{n_k}})
\end{equation*}
one gets that $(u^*,v^*)\in S_{(z_1,z_2)}$.

The proof is completed by showing that $(u^{\ast },v^{\ast })$ is the
smallest solution of (\ref{pz}) within $[\underline{u},\overline{u}]\times
\lbrack \underline{v},\overline{v}]$. Indeed, fix a solution $(w_{1},w_{2})$
of (\ref{pz}) such that $(w_{1},w_{2})\in \lbrack \underline{u},\overline{u}%
]\times \lbrack \underline{v},\overline{v}]$. We can conclude verifying that 
\begin{equation}
u^{\ast }\leq w_{1},\quad v^{\ast }\leq w_{2}.  \label{u<w}
\end{equation}%
According to Lemma \ref{L4}, there exists $(w_{1,n},w_{2,n})\in
S_{(z_{1,n},z_{2,n})}$ such that 
\begin{equation}
(w_{1,n},w_{2,n})\rightarrow (w_{1},w_{2})\text{ in }C_{0}^{1}(\overline{%
\Omega })\times C_{0}^{1}(\overline{\Omega })\text{ \ as }n\rightarrow
\infty .  \label{w1n,w2n}
\end{equation}%
Then, since $(u_{n_{k}}^{\ast },v_{n_{k}}^{\ast })$ is the smallest solution
in $[\underline{u},\overline{u}]\times \lbrack \underline{v},\overline{v}]$
of $(P_{(z_{1,n_{k}},z_{2,n_{k}})})$, it is clear that 
\begin{equation*}
u_{n_{k}}^{\ast }\leq w_{1,n_{k}},\quad v_{n_{k}}^{\ast }\leq w_{2,n_{k}},
\end{equation*}%
for all $k\in \mathbb{N}$. Passing to the limit in the previous inequalities
and bearing in mind (\ref{subsequence}) and (\ref{w1n,w2n}) one directly
achieves (\ref{u<w}). This ends the proof.
\end{proof}

\begin{lemma}
\label{L3} $\mathcal{T}(\mathcal{K}_{1}(C)\times \mathcal{K}_{2}(C))\subset 
\mathcal{K}_{1}(C)\times \mathcal{K}_{2}(C)$ provided $C>1$ is large enough.
\end{lemma}

\begin{proof}
For any $(z_{1},z_{2})\in \mathcal{K}_{1}(C)\times \mathcal{K}_{2}(C)$, let $%
(u^{\ast },v^{\ast })_{(z_{1},z_{2})}=\mathcal{T}(z_{1},z_{2})$ be the
smallest solution of (\ref{pz}) in $[\underline{u},\overline{u}]\times
\lbrack \underline{v},\overline{v}]$. Then, according to $\mathrm{H}(f)$, $%
\mathrm{H}(g)$, (\ref{11}), (\ref{21}) and (\ref{9}), we get%
\begin{eqnarray*}
f(x,u^{\ast },v^{\ast },\nabla z_{1},\nabla z_{2})& \leq & M_{1}(u^{\ast
})^{\alpha _{1}}(v^{\ast })^{\beta _{1}}+|\nabla z_{1}|^{\gamma
_{1}}+|\nabla z_{2}|^{\theta _{1}} \\
& \leq & M_{1}\underline{u}^{\alpha _{1}}\overline{v}^{\beta
_{1}}+(CR_{1})^{\gamma _{1}}+(CR_{2})^{\theta _{1}} \\
& \leq & M_{1}C^{-\alpha _{1}+\beta _{1}}(ld(x))^{\alpha _{1}}(c_{1}^{\prime
}d(x))^{\beta _{1}}+(CR_{1})^{\gamma _{1}}+(CR_{2})^{\theta _{1}} \\
& \leq & M_{1}C^{-\alpha _{1}+\beta _{1}}l^{\alpha _{1}}(c_{1}^{\prime
})^{\beta _{1}}\left\Vert d(x)\right\Vert _{\infty }^{\alpha _{1}+\beta
_{1}}+(CR_{1})^{\gamma _{1}}+(CR_{2})^{\theta _{1}} \\
& \leq & (K_{p}^{-1}CR_{1})^{p-1}
\end{eqnarray*}%
a.e. in $\Omega$ and, analogously, 
\begin{eqnarray*}
g(x,u^{\ast },v^{\ast },\nabla z_{1},\nabla z_{2})& \leq & M_{2}(u^{\ast
})^{\alpha _{2}}(v^{\ast })^{\beta _{2}}+|\nabla z_{1}|^{\gamma
_{2}}+|\nabla z_{2}|^{\theta _{2}} \\
& \leq & M_{2}\overline{u}^{\alpha _{2}}\underline{v}^{\beta
_{2}}+(CR_{1})^{\gamma _{2}}+(CR_{2})^{\theta _{2}} \\
& \leq & M_{2}C^{\alpha _{2}-\beta _{2}}\left\Vert \xi _{1}\right\Vert
_{\infty }^{\alpha _{2}}\phi _{1,q}^{\beta _{2}}+(CR_{1})^{\gamma
_{2}}+(CR_{2})^{\theta _{2}} \\
& \leq & M_{2}C^{\alpha _{2}-\beta _{2}}c_{1}^{\alpha _{2}}l^{\beta
_{2}}\left\Vert d(x)\right\Vert _{\infty }^{\alpha _{2}+\beta
_{2}}+(CR_{1})^{\gamma _{2}}+(CR_{2})^{\theta _{2}} \\
& \leq & (K_{q}^{-1}CR_{2})^{q-1}
\end{eqnarray*}
a.e. in $\Omega$, provided that $C>1$ is sufficiently large. Then, using the
inequality in (\ref{12}), it follows that%
\begin{equation*}
\begin{array}{c}
\left\Vert \nabla u^{\ast }\right\Vert _{\infty }\leq CR_{1}\text{ \ and \ }%
\left\Vert \nabla v^{\ast }\right\Vert _{\infty }\leq CR_{2},%
\end{array}%
\end{equation*}%
namely $(u^{\ast },v^{\ast })_{(z_{1},z_{2})}\in \mathcal{K}_{1}(C)\times 
\mathcal{K}_{2}(C)$. This ends the proof of lemma.
\end{proof}

Now we are in position to prove our main result.

\begin{proof}[Proof of Theorem \protect\ref{T1}]
On the basis of Lemmas \ref{L2} and \ref{L3}, Schauder's fixed point theorem
(see, e.g., \cite[p. 57]{Z}) garantees the existence of $(u,v)\in \mathcal{K}%
(C)_{1}\times \mathcal{K}_{2}(C)$ satisfying $(u,v)=\mathcal{T}(u,v).$
Taking into account the definition of $\mathcal{T}$, it turns out that $%
(u,v)\in C_{0}^{1}(\overline{\Omega })\times C_{0}^{1}(\overline{\Omega })$
is a (positive) solution of problem (\ref{p}). Since $(u,v)\in \mathcal{K}%
_{1}(C)\times \mathcal{K}_{2}(C)$, in particular, $(u,v)\in [\underline
u,\overline u]\times[\underline v,\overline v]$ and on account of (\ref{11}%
), (\ref{21}) and (\ref{9}), the property (\ref{c}) is fullfield. This
completes the proof.
\end{proof}

\begin{acknowledgement}
This work was partially performed when the third-named author visited Reggio
Calabria University, to which he is grateful for the kind hospitality.
\end{acknowledgement}

\end{document}